\documentclass[a4paper,10pt]{amsart}

\usepackage[utf8]{inputenc}
\usepackage{amsfonts, amsmath, amssymb, amsthm, mathtools}
\usepackage{xcolor}
\definecolor{darkblue}{rgb}{0,0,0.6}
\usepackage[ocgcolorlinks,colorlinks=true, citecolor=darkblue, filecolor=darkblue, linkcolor=darkblue, urlcolor=darkblue]{hyperref}
\usepackage[nameinlink, capitalise]{cleveref}
\usepackage{tikz}
\usetikzlibrary{matrix,arrows,trees}
\usepackage{tikz-cd}

\swapnumbers

\numberwithin{equation}{section}

\newtheorem{thm}[equation]{Theorem}
\newtheorem{prop}[equation]{Proposition}
\newtheorem{lem}[equation]{Lemma}
\newtheorem{cor}[equation]{Corollary}
\theoremstyle{definition}
\newtheorem{rem}[equation]{Remark}

\newtheorem{const}[equation]{Construction}
\newtheorem{defn}[equation]{Definition}
\newtheorem{ex}[equation]{Example}

\crefname{lem}{Lemma}{Lemmas}
\crefname{prop}{Proposition}{Propositions}

\newcounter{commentcounter}

\renewcommand{\epsilon}{\varepsilon}
\renewcommand{\theta}{\vartheta}
\renewcommand{\phi}{\varphi}
\renewcommand{\hat}{\widehat}

\newcommand{\bbN}{\mathbb{N}}

\newcommand{\bbZ}{\mathbb{Z}}

\newcommand{\bfK}{\mathbf{K}}

\newcommand{\bfN}{\mathbf{N}}

\newcommand{\bfV}{\mathbf{V}}
\newcommand{\bfW}{\mathbf{W}}

\newcommand{\cA}{\mathcal{A}}
\newcommand{\cB}{\mathcal{B}}
\newcommand{\cC}{\mathcal{C}}
\newcommand{\cD}{\mathcal{D}}
\newcommand{\cE}{\mathcal{E}}
\newcommand{\cF}{\mathcal{F}}

\newcommand{\cM}{\mathcal{M}}

\newcommand{\cU}{\mathcal{U}}

\DeclareMathOperator*{\colim}{colim}
\DeclareMathOperator{\cofib}{cofib}
\DeclareMathOperator{\fib}{fib}
\DeclareMathOperator{\id}{id}

\DeclareMathOperator{\Fun}{Fun}
\DeclareMathOperator{\sd}{sd}
\DeclareMathOperator{\grad}{gr}
\DeclareMathOperator{\Map}{Map}
\DeclareMathOperator{\Idem}{Idem}
\DeclareMathOperator{\Ind}{Ind}

\DeclareMathOperator{\grayson}{\Gamma}

\newcommand{\catex}{\mathrm{Cat}^{\mathrm{ex}}_\infty}
\newcommand{\catperf}{\mathrm{Cat}^{\mathrm{perf}}_\infty}
\newcommand{\madd}{\cM_{\mathrm{add}}}
\newcommand{\mloc}{\cM_{\mathrm{loc}}}
\newcommand{\uadd}{\cU_{\mathrm{add}}}
\newcommand{\uloc}{\cU_{\mathrm{loc}}}
\newcommand{\Sp}{\mathrm{Sp}}
\newcommand{\Ab}{\mathrm{Ab}}
\newcommand{\gapindex}[1]{#1^{[1]}}
\newcommand{\disc}{\delta}
\newcommand{\Kcon}{K}
\newcommand{\topbot}{\mathrlap{\top}\bot}

\title[Algebraic $K$--theory of stable $\mathbf{\infty}$--categories via binary complexes]{Algebraic \boldmath $K$--theory of stable $\mathbf{\infty}$--categories \unboldmath via binary complexes}
\author{Daniel Kasprowski}
\address{Rheinische Friedrich-Wilhelms-Universit\"at Bonn, Mathematisches Institut,\newline\indent Endenicher Allee 60, 53115 Bonn, Germany}
\email{kasprowski@uni-bonn.de}
\urladdr{http://www.math.uni-bonn.de/people/daniel/} 

\author{Christoph Winges}
\address{Rheinische Friedrich-Wilhelms-Universit\"at Bonn, Mathematisches Institut,\newline\indent Endenicher Allee 60, 53115 Bonn, Germany}
\email{winges@math.uni-bonn.de}
\urladdr{http://www.math.uni-bonn.de/people/winges/} 


\begin{document}

\begin{abstract}
We adapt Grayson's model of higher algebraic $K$--theory using binary acyclic complexes to the setting of stable $\infty$--categories.
As an application, we prove that the $K$--theory of stable $\infty$--categories preserves infinite products.
\end{abstract}

\maketitle

\section{Introduction}\label{sec:intro}
Using binary acyclic complexes, Grayson \cite{Grayson2012} gave a description of higher algebraic $K$--theory for exact categories in terms of generators and relations.
The present article shows that Grayson's picture of higher algebraic $K$--theory admits a concise description in the context of stable $\infty$--categories, using the language introduced by Blumberg--Gepner--Tabuada in \cite{BGT13}; see \cref{sec:alg-k} for a quick recollection of the notions of \emph{additive} and \emph{localizing invariants} as well as the associated categories of motives.
Let $\uadd$ and $\uloc$ denote the universal additive and localizing invariants from \cite{BGT13}.

In \cref{sec:model}, we define analogs $F^q\cC$ and $B^q\cC$ of acyclic complexes and Grayson's binary acyclic complexes in the setting of stable $\infty$--categories.
Following Grayson's argument, we obtain the following analog of \cite[Corollary~6.5]{Grayson2012}.

\begin{thm}
\label{thm:mainintro}
 The cofiber $\grayson\cC$ of the `diagonal' $\uloc(\Delta) \colon \uloc(F^q\cC) \to \uloc(B^q\cC)$ is naturally equivalent to $\Omega\uloc(\cC)$.
\end{thm}

Each binary acyclic complex has a length.
One may ask to which extent the abelian group $\pi_0\bfK(\grayson_r\cC)$ defined in terms of binary acyclic complexes of a fixed length $r$ differs from the abelian group $\pi_0\bfK(\grayson\cC)$.
Following the construction from \cite{KasprowskiWinges}, we show in \cref{sec:shortening} that binary acyclic complexes of length~$5$ generate the whole group $\pi_0\bfK(\grayson\cC)$, and we construct a split to binary acyclic complexes of length~$7$. More explicitly, we show the following.

\begin{thm}
	\label{thm:split2}
	Let $\cC$ be an idempotent complete stable $\infty$--category.
	The canonical map $\pi_0\bfK(\grayson_5\cC)\to \pi_0\bfK(\grayson\cC)$ is a surjection and the canonical map $\pi_0\bfK(\grayson_7\cC) \to \pi_0\bfK(\grayson\cC) \cong K_1(\cC)$ admits a natural section.
\end{thm}

The $K$--theory functor preserves filtered colimits and finite products for essentially trivial reasons. By work of Carlsson \cite{Carlsson95}, it is known that the algebraic $K$--theory functor commutes with infinite products of exact categories and of Waldhausen categories with a cylinder functor. This result plays an essential role in proofs of the integral Novikov conjecture, e.g.\ for virtually polycyclic groups \cite{CP,rosenthal} and linear groups with a finite dimensional model for the classifying space for proper actions \cite{RTY,KasFDC,KasLin}. As an application of \cref{thm:split2}, we will imitate the proof of Carlsson's theorem given in \cite{KasprowskiWinges} and show the following result in \cref{sec:products}. 
\begin{thm}
	\label{thm:commute}
	For every family $\{\cC_i\}_{i\in I}$ of small stable $\infty$-categories, the natural map
	\[\bfK\left(\prod_{i\in I}\cC_i\right)\to \prod_{i\in I} \bfK(\cC_i)\]
	is an equivalence. The same result holds for connective algebraic $K$--theory $\Kcon$ instead of non-connective $K$--theory $\bfK$.
\end{thm}

The case of connective algebraic $K$--theory admits the following generalization to the universal additive invariant.

\begin{thm}\label{thm:uadd-products}
 The universal additive invariant $\uadd \colon \catex \to \madd$ commutes with arbitrary products.
 In particular, any additive invariant that becomes corepresentable on the category of non-commutative motives $\madd$ commutes with arbitrary products.
\end{thm}

\subsection*{Acknowledgements}
We thank Akhil Mathew for helpful discussions and the referee for useful comments on a previous version of the article. The authors are members of the Hausdorff Center for Mathematics at the University of Bonn. The second author furthermore acknowledges support by the Max Planck Society and by Wolfgang L\"uck's ERC Advanced Grant ``KL2MG-interactions" (no.~662400).

\section{Algebraic \texorpdfstring{$K$--theory}{K-theory} of \texorpdfstring{stable $\infty$--categories}{stable infinity-categories}}\label{sec:alg-k}

We understand algebraic $K$--theory as an invariant of stable $\infty$--categories in the sense of Lurie \cite{LurieHA}.
In this context, the work of Blumberg--Gepner--Tabuada \cite{BGT13} provides a conceptual approach to algebraic $K$--theory which, in addition to some elementary facts about the class group $K_0$, is sufficient for the purposes of this article.
The current section serves to recall the necessary statements and terminology.

\subsection{Algebraic \texorpdfstring{$K$--theory}{K-theory} via universal invariants}
We work in the setting of quasicategories, the term $\infty$--category being used synonymously throughout.
A quasicategory is \emph{stable} if it admits all finite limits and colimits, and pushout squares and pullback squares coincide \cite[Proposition~1.1.3.4]{LurieHA}.
A functor between stable $\infty$--categories is \emph{exact} if it preserves finite limits and colimits \cite[Proposition~1.1.4.1]{LurieHA}.
We denote by $\catex$ the $\infty$--category of small stable $\infty$--categories and exact functors.
Inside $\catex$ we have the full subcategory of idempotent complete stable $\infty$--categories $\catperf$.
The idempotent completion of any small stable $\infty$--category is stable \cite[Corollary~1.1.3.7]{LurieHA}, and the idempotent completion functor $\Idem \colon \catex \to \catperf$ provides a left adjoint to the inclusion functor.

We are mainly interested in \emph{localizing} and \emph{additive} invariants in the sense of Blumberg--Gepner--Tabuada \cite{BGT13}.
In order to give the definitions of these notions, we recall some additional terminology.

\begin{defn}
	\label{def:quotient}
	Let $\cA$ and $\cB$ be stable $\infty$--categories, and let $\cA \subseteq \cB$ be a full subcategory.
	We define the Verdier quotient $\cB/\cA$ as the localization of $\cB$ at the collection of arrows in $\cB$ whose cofiber is equivalent to an object in $\cA$.
\end{defn}

\begin{defn}[{cf.~\cite[Definition~5.12 and Proposition~5.13]{BGT13}}]
 \label{def:exact-sequence}
 An \emph{exact sequence} in $\catex$ is a sequence of exact functors
 \[ \cA \to \cB \to \cC \]
 satisfying the following properties:
 \begin{enumerate}
  \item The functor $\cA \to \cB$ is fully faithful.
  \item The composed functor $\cA \to \cC$ is trivial.
  \item The induced exact functor $\Idem(\cB/\cA) \to \Idem(\cC)$ is an equivalence.
 \end{enumerate}
\end{defn}

Our definition of an exact sequence of stable $\infty$--categories does not match exactly the definition and characterization in \cite{BGT13}.
However, \cite[Proposition~I.3.5]{NikolausScholze} shows that \cref{def:exact-sequence} is equivalent to the one in \cite{BGT13}.
We prefer the present formulation because it applies more directly in the arguments of \cref{sec:model}. 

\begin{defn}[{\cite[Definition~8.1]{BGT13}}]
Let $\cD$ be a presentable stable $\infty$--category.
A functor $F \colon \catex \to \cD$ is a \emph{localizing invariant} if it satisfies the following properties:
\begin{enumerate}
 \item $F$ commutes with filtered colimits.
 \item\label{item:cofib} $F$ sends exact sequences to cofiber sequences. 
\end{enumerate} 
\end{defn}
\begin{rem}
	Note that \cite[Definition~8.1]{BGT13} also requires that $F$ sends the canonical functor $\cC \to \Idem(\cC)$ to an equivalence for every $\cC \in \catex$. But this is a consequence of \ref{item:cofib} by considering $0\to \cC\to \Idem(\cC)$.
\end{rem}

\begin{thm}[{\cite[Theorem~8.7]{BGT13}}]
 There exists a presentable stable $\infty$--category $\mloc$ and a localizing invariant $\uloc \colon \catex \to \mloc$ such that
 \[ (\uloc)^* \colon \Fun^{\mathrm{Lex}}(\mloc,\cD) \to \Fun^{\mathrm{loc}}(\catex,\cD) \]
 is an equivalence for every presentable stable $\infty$--category $\cD$, where $\Fun^{\mathrm{Lex}}$ denotes the $\infty$--category of colimit-preserving functors,
 and $\Fun^{\mathrm{loc}}$ denotes the $\infty$--category of localizing invariants.
\end{thm}

For the moment, we contend ourselves with the characterization of non-connective algebraic $K$--theory $\bfK \colon \catex \to \Sp$
as the localizing invariant corepresented by $\uloc(\Sp^\omega)$, where $\Sp^\omega$ denotes the stable $\infty$--category of compact spectra \cite[Theorem~9.8]{BGT13}.
That is,
	\begin{equation}\label{eq:corep}
	\bfK(\cA)\simeq \Map(\uloc(\Sp^\omega),\uloc(\cA)).
	\end{equation}
The notion of an additive invariant is closely related to that of a localizing invariant.
The difference lies in the fact that we require fewer exact sequences to be turned into cofiber sequences.

\begin{defn}[{\cite[Definition~5.18]{BGT13}}]
 An exact sequence $\cA \xrightarrow{f} \cB \xrightarrow{g} \cC$ of stable $\infty$--categories is \emph{split-exact}
 if $f$ and $g$ admit right adjoints $i \colon \cB \to \cA$ and $j \colon \cC \to \cB$, respectively,
 such that the unit $\id \to if$ and counit $gj \to \id$ are equivalences.
\end{defn}
\begin{rem}
	Note that in the above definition the unit $\id\to if$ is automatically an equivalence since $f$ is fully faithful.
	One can also prove that the counit $gj\to \id$ is an equivalence by showing that $j$ is fully faithful.
\end{rem}

\begin{ex}\label{ex:classical-additivity}
 The following is the canonical non-trivial example of a split-exact sequence of stable $\infty$--categories.
 Let $\cC$ be a stable $\infty$--category, and consider the stable $\infty$--category $\cE(\cC)$ of cofiber sequences in $\cC$.
 Then $\cE(\cC)$ fits into a split-exact sequence
 \[ \cC \xrightarrow{k} \cE(\cC) \xrightarrow{q} \cC, \]
 in which the functor $q$ projects to the cofiber, while $k$ sends an object $X$ to the canonical cofiber sequence $X \xrightarrow{\id} X \to 0$.
\end{ex}

\begin{defn}[{\cite[Definition~6.1]{BGT13}}]
Let $\cD$ be a presentable stable $\infty$--category.
A functor $F \colon \catex \to \cD$ is an \emph{additive invariant} if it satisfies the following properties:
\begin{enumerate}
 \item $F$ commutes with filtered colimits.
 \item $F$ sends the canonical functor $\cC \to \Idem(\cC)$ to an equivalence for every $\cC \in \catex$.
 \item For every split-exact sequence
  \[\begin{tikzcd} \cA \ar[r, "f", shift left] & \cB\ar[l, "i", shift left]\ar[r, "g", shift left] & \cC\ar[l, "j", shift left] \end{tikzcd}\]
  in $\catex$, the morphism $F(f) + F(j) \colon F(\cA) \oplus F(\cC) \to F(\cB)$ is an equivalence.
\end{enumerate}
\end{defn}

\begin{thm}[{\cite[Theorem~6.10]{BGT13}}]
 There exists a presentable stable $\infty$--category $\madd$ and an additive invariant $\uadd \colon \catex \to \madd$ such that
 \[ (\uadd)^* \colon \Fun^{\mathrm{Lex}}(\madd,\cD) \to \Fun^{\mathrm{add}}(\catex,\cD) \]
 is an equivalence for every presentable stable $\infty$--category $\cD$, where $\Fun^{\mathrm{Lex}}$ denotes the $\infty$--category of colimit-preserving functors,
 and $\Fun^{\mathrm{add}}$ denotes the $\infty$--category of additive invariants.
\end{thm}

Assuming the non-connective algebraic $K$--theory functor $\bfK$ to be known,
connective algebraic $K$--theory $\Kcon \colon \catex \to \Sp$ can be characterized as the functor $\tau_{\geq 0} \circ \bfK$,
where $\tau_{\geq 0} \colon \Sp \to \Sp$ denotes the functor taking connective covers.
By \cite[Theorem~7.13]{BGT13}, connective algebraic $K$--theory, considered as a colimit-preserving functor on $\madd$, is corepresented by $\uadd(\Sp^\omega)$.
Note that with the above definition, the connective algebraic $K$--theory of $\cC$ only depends on $\Idem(\cC)$. It is possible to define connective algebraic $K$-theory such that the values at $\cC$ and $\Idem(\cC)$ may differ, but here we are following \cite{BGT13}.

\subsection{Lower algebraic \texorpdfstring{$K$--theory}{K-theory}}\label{sec:lower-k-theory}

Our arguments do require some additional information about the lower algebraic $K$--groups of a small stable $\infty$--category.
For $\cC$ a small stable $\infty$--category, we denote by $K_0(\cC)$ the abelian group generated by equivalence classes of objects in $\cC$,
subject to the relation $[C_0] + [C_2] = [C_1]$ for every cofiber sequence $C_0 \to C_1 \to C_2$ in $\cC$.
In general, we have an isomorphism
\[ K_0(\Idem(\cC)) \cong \pi_0\Kcon(\cC) \cong \pi_0\bfK(\cC). \]
The next lemma gives a convenient criterion for equality of elements in $K_0(\cC)$.

\begin{lem}[Heller's criterion]\label{lem:hellers-criterion}
 Let $X$ and $Y$ be objects in a stable $\infty$--category $\cC$.
 
 Then $[X] = [Y] \in K_0(\cC)$ if and only if there are cofiber sequences
 \[A \to X \oplus S \to B \quad \text{and} \quad A \to Y \oplus S \to B \]
 for some objects $A$, $B$, $S \in \cC$.
\end{lem}
\begin{proof}
 The proof of \cite[Lemma~3.2]{KasprowskiWinges} applies almost verbatim.
\end{proof}

Fix an uncountable regular cardinal $\kappa$.
The \emph{suspension} of a stable $\infty$--category $\cC$ is defined to be
\[ \Sigma\cC := \Ind_{\aleph_0}(\cC)^\kappa/\cC, \]
where $\Ind_{\aleph_0}(\cC)^\kappa$ denotes the full subcategory of $\kappa$-compact objects in the $\Ind$-completion of $\cC$, see also \cite[Section~2.4 and Section~9.1]{BGT13}.
The stable $\infty$--category $\Ind_{\aleph_0}(\cC)^\kappa$ admits a canonical Eilenberg swindle since it admits countable coproducts, hence has contractible $K$--theory.
Non-connective $K$--theory can be constructed from connective $K$--theory by the formula
\[ \bfK(\cC) \simeq \colim_{n \in \bbN} \Omega^n\Kcon(\Sigma^n\cC), \]
see \cite[Definition~9.6]{BGT13}.
It follows that $\bfK(\cC) \simeq \Omega^n\bfK(\Sigma^n\cC)$ for every $n \in \bbN$,
and hence the negative $K$--groups of any stable $\infty$--category can be described by the formula
\[ \pi_{-n}\bfK(\cC) \cong K_0(\Idem(\Sigma^n\cC)). \]

\section{Grayson's model for \texorpdfstring{stable $\infty$--categories}{stable infinity-categories}}
\label{sec:model}

Let $\cC$ be a stable $\infty$--category.

\begin{defn}[{cf.~\cite[Definition~1.2.2.2]{LurieHA}}]\label{defn:filtered-object}
 For a linearly ordered set $I$, denote by $\gapindex{I}$ the subposet of $I \times I$ containing all pairs $(i,j)$ satisfying $i \leq j$.
 An \emph{$I$--complex in $\cC$} is a functor $X \colon N(\gapindex{I}) \to \cC$ such that $X_{i,i}$ is a zero object for all $i$,
 and such that for $i \leq j \leq k$ the square
 \[\begin{tikzcd}
    X_{i,j}\ar[r]\ar[d] & X_{i,k}\ar[d] \\
    X_{j,j}\ar[r] & X_{j,k}
   \end{tikzcd} \]
 is a pushout.
 Denote by $S_I\cC$ the full subcategory of $\Fun(\gapindex{I},\cC)$ spanned by the $I$--complexes.
 For $I = [n]$, we also write $S_n\cC$ instead of $S_{[n]}\cC$.
  
 An $\bbN$--complex $X$ is called \emph{bounded} if there exists a natural number $r \in \bbN$ such that $X_{0,i} \to X_{0,i+1}$ is an equivalence for $i \geq r$.
 In this case, we say that $X$ \emph{is supported on $[0,r]$}.
 Let $F\cC$ denote the full subcategory of $S_\bbN\cC$ spanned by the bounded $\bbN$--complexes.
 We call $F\cC$ the stable $\infty$--category of \emph{bounded complexes}.
 
 Let $F^q\cC \subseteq F\cC$ be the full stable subcategory spanned by those bounded complexes $X$ such that $X_{0,i} \simeq 0$ for $i \gg 0$.
\end{defn}

\begin{defn}
 Let $\bbN^\disc$ denote the set of natural numbers considered as a discrete poset.
 The map $\bbN^\disc \to \gapindex{\bbN},\ i \mapsto (i,i+1)$ induces the functor
 \[ \grad \colon F\cC \to \bigoplus_{\bbN} \cC \]
 which associates to a bounded complex its \emph{(underlying) graded object}.
\end{defn}

\begin{defn}\label{defn:binary-complexes}
 The category of \emph{binary complexes} $B\cC$ is defined by the following pullback square of $\infty$--categories:
 \[\begin{tikzcd}
    B\cC\ar{r}{\top}\ar{d}[swap]{\bot} & F\cC\ar{d}{\grad} \\
    F\cC\ar{r}{\grad} & \bigoplus\limits_{\bbN} \cC
  \end{tikzcd}\]
 The functor $\top$ is called the \emph{top projection}, while $\bot$ is the \emph{bottom projection}.
 By substituting $F^q\cC$ at appropriate places, we obtain the following variations of $B\cC$:
 \[\begin{tikzcd}
    B^t\cC\ar{r}{\top}\ar{d}[swap]{\bot} & F^q\cC\ar{d}{\grad} \\
    F\cC\ar{r}{\grad} & \bigoplus\limits_{\bbN} \cC
  \end{tikzcd}\qquad
  \begin{tikzcd}
    B^b\cC\ar{r}{\top}\ar{d}[swap]{\bot} & F\cC\ar{d}{\grad} \\
    F^q\cC\ar{r}{\grad} & \bigoplus\limits_{\bbN} \cC
  \end{tikzcd}\qquad
  \begin{tikzcd}
    B^q\cC\ar{r}{\top}\ar{d}[swap]{\bot} & F^q\cC\ar{d}{\grad} \\
    F^q\cC\ar{r}{\grad} & \bigoplus\limits_{\bbN} \cC
  \end{tikzcd}\]
 The objects of $B^t\cC$, $B^b\cC$ and $B^q\cC$ are called \emph{binary $\top$--acyclic complexes}, \emph{binary $\bot$--acyclic complexes} and \emph{binary acyclic complexes}, respectively.
\end{defn}
 
\begin{rem}\label{rem:bc-explicit-description}
 Occasionally, we will require an explicit description of maps $K \to B\cC$, where $K$ is some simplicial set.
 Since limits of stable $\infty$--categories may be computed in the category of all $\infty$--categories \cite[Theorem~1.1.4.4]{LurieHA},
 it follows from \cite[Example~17.7.3 and Remark~17.7.4]{MR3221774} that maps $p \colon K \to B\cC$ may be described
 by a pair $(p^\top, p^\bot)$ of maps $K \to F\cC$ together with a natural equivalence $\grad \circ p^\top \xrightarrow{\sim} \grad \circ p^\bot$ in $\bigoplus_\bbN \cC$.
 To specify a map $K \to B^t\cC$, $K \to B^b\cC$ or $K \to B^q\cC$, we additionally have to require that $p^\top$, $p^\bot$ or both of them map to the full subcategory $F^q\cC$.
\end{rem}

Note that $B^q\cC$, $B^t\cC$, $B^b\cC$ and $B\cC$ are all stable.
Moreover, these categories fit into a commutative square
\[\begin{tikzcd}
  B^q\cC\ar{r}\ar{d} & B^t\cC\ar{d} \\
  B^b\cC\ar{r} & B\cC
\end{tikzcd}\]
in which all functors are fully faithful and exact.
The identity functor on $F\cC$ induces the \emph{diagonal functor} $\Delta \colon F\cC \to B\cC$.
Similarly, the identity functor on $F^q\cC$ induces a functor $\Delta \colon F^q\cC \to B^q\cC$,
and these fit into a commutative square of exact functors
\[\begin{tikzcd}
  F^q\cC\ar{r}{\Delta}\ar{d} & B^q\cC\ar{d} \\
  F\cC\ar{r}{\Delta} & B\cC
\end{tikzcd}\]
By definition, the diagonal functor $\Delta$ is split by both $\top$ and $\bot$.

\begin{defn}\label{defn:k-omega-c}
 Let $\cC$ be a small stable $\infty$--category.
 Define the \emph{Grayson construction} $\grayson\cC$ on $\cC$ to be the motive
 \[ \grayson\cC := \cofib( \uloc(F^q\cC) \xrightarrow{\uloc(\Delta)} \uloc(B^q\cC) ). \]
\end{defn}

Our goal is to show that $\grayson\cC$ represents $\Omega\uloc(\cC)$.
The proof follows closely Grayson's arguments in \cite{Grayson2012}.
Note that the functor $\Delta\colon F^q\cC\to B^q\cC$ is not fully faithful and hence we cannot directly form the Verdier quotient in $\catex$.

Objects in the $\infty$--category $F\cC$ are filtered objects in $\cC$ with a choice of all possible filtration quotients.
To make some of the upcoming proofs easier to read, we also define versions of these categories which do not include choices of filtration quotients,
which will make it easier to map into these categories.

Let $f\cC \subseteq \Fun(N(\bbN), \cC)$ be the full subcategory of \emph{bounded filtrations},
i.e.~the full subcategory spanned by those functors $X$ which are essentially constant in all but finitely many degrees and satisfy $X_0 \simeq 0$.
By \cite[Lemma~1.2.2.4]{LurieHA}, the forgetful functor $u \colon F\cC \to f\cC$ induced by the map $\bbN \to \gapindex{\bbN},\ i \mapsto (0,i)$ is an equivalence.

Denote by $f^q\cC \subseteq f\cC$ the full stable subcategory spanned by those $X$ satisfying $\colim X \simeq 0$.

By choosing an inverse of the equivalence $u \colon F\cC \xrightarrow{\sim} f\cC$, we also obtain a functor $\grad \colon f\cC \to \bigoplus_\bbN \cC$.
This allows us to define stable $\infty$--categories $b\cC$, $b^t\cC$, $b^b\cC$ and $b^q\cC$ in analogy to \cref{defn:binary-complexes}.

The next lemma, which is reminiscent of the Gillet--Waldhausen theorem, but much easier to prove,
tells us that we may concentrate on describing the $K$--theory of $F\cC/F^q\cC$.
Let $\iota \colon \cC \to f\cC$ denote the functor induced by the projection $N(\bbN) \to \Delta^0$,
and let $\pi \colon f\cC \to \cC$ denote the colimit functor (which exists since filtrations in $f\cC$ become essentially constant).
The structure maps of the colimit provide a natural transformation $\tau \colon \id \to \iota \circ \pi$.

\begin{lem}\label{lem:gillet-waldhausen}
 The functors $\overline{\iota} \colon \cC \to f\cC/f^q\cC$ and $\overline{\pi} \colon f\cC/f^q\cC \to \cC$ induced by $\iota$ and $\pi$ are equivalences.
\end{lem}
\begin{proof}
 There is an evident equivalence $\pi \circ \iota \simeq \id_\cC$.
 Since $\pi$ vanishes on $f^q\cC$, there is an induced exact functor $\overline{\pi} \colon f\cC/f^q\cC \to \cC$
 as well as a natural transformation $\overline{\tau} \colon \id \to \overline{\iota} \circ \overline{\pi}$.
 We still have $\overline{\pi} \circ \overline{\iota} \simeq \id$.
 Moreover, $\overline{\tau}$ is a natural equivalence by \cref{def:quotient} of the Verdier quotient since $\cofib(\tau_X)$ lies in $f^q\cC$ for every $X \in f\cC$.
 The claim follows.
\end{proof}
 
In what follows, we will omit the overline decoration on $\iota$ and $\pi$; it should always be clear from context whether we consider these as functors to/from the quotient category or the original category.
 
\begin{lem}\label{lem:grayson-as-fiber}
 There is a natural equivalence
 \[ \grayson\cC \simeq \fib( \uloc(B^q\cC) \xrightarrow{\uloc(\top)} \uloc(F^q\cC) ). \]
\end{lem}
\begin{proof}
 Since $\top \circ \Delta \simeq \id$, we have the following commutative diagram, natural in $\cC$:
 \[\begin{tikzcd}
    0\ar[r]\ar[d] & \uloc(F^q\cC)\ar[r, "\id"]\ar[d, "\uloc(\Delta)"] & \uloc(F^q\cC)\ar[d, "\id"] \\
    \fib( \uloc(B^q\cC) \xrightarrow{\uloc(\top)} \uloc(F^q\cC) )\ar[r]\ar[d, "\id"] & \uloc(B^q\cC)\ar[r, "\uloc(\top)"]\ar[d] & \uloc(F^q\cC)\ar[d] \\
    \fib( \uloc(B^q\cC) \xrightarrow{\uloc(\top)} \uloc(F^q\cC) )\ar[r, "\simeq"] & \grayson\cC\ar[r] & 0
   \end{tikzcd}\]
\end{proof}

\begin{defn}\label{defn:filtrations-fixed-length}
 Let $f_r\cC \subseteq f\cC$ denote the full subcategory of bounded filtrations supported on $[0,r]$,
 and define $b_r\cC$ as the pullback
 \[\begin{tikzcd}
    b_r\cC\ar{r}{\top}\ar{d}[swap]{\bot} & f_r\cC\ar{d}{\grad} \\
    f_r\cC\ar{r}{\grad} & \bigoplus\limits_{\bbN} \cC
  \end{tikzcd}\]
\end{defn}

\begin{defn}\label{def:shifted-constant-filtration}
 Let $p_r\colon N(\bbN)\to \Delta^1$ denote the map characterized by sending $i$ to $0$ for $i\leq r$ and $i$ to $1$ for $i\geq r+1$.
 Furthermore, choose a zero object $0$ in $\cC$ and a section $s$ of the trivial fibration $\cC_{0/}\to \cC$.
 Define the functor $\iota_{r+1}\colon \cC\to f_{r+1}\cC$ as the composition
 \[\cC\xrightarrow{s}\cC_{0/}\subseteq \Fun(\Delta^1,\cC)\xrightarrow{p_r^*}f_{r+1}\cC.\]
\end{defn}

\begin{lem}\label{lem:fc-bc}
The map $\uloc(\top) \colon \uloc(B\cC) \to \uloc(F\cC)$ induced by the top projection is an equivalence.
\end{lem}
\begin{proof}
 Since we have a commutative square
 \[\begin{tikzcd}
    B\cC\ar[r, "\top"]\ar[d, "\simeq"] & F\cC\ar[d, "\simeq"] \\
    b\cC\ar[r, "\top"] & f\cC
   \end{tikzcd}\]
 it suffices to prove the claim for the lower horizontal arrow.
 
 Since $f\cC \simeq \colim_r f_r\cC$ and finite limits commute with directed colimits
 (combine \cite[Theorem~1.1.4.4 and Proposition~1.1.4.6]{LurieHA} with \cite[Lemma~5.4.5.6]{MR2522659}),
 we also have $b\cC \simeq \colim_r b_r\cC$.
 As $\uloc$ commutes with filtered colimits, it is enough to prove that $\uloc(\top) \colon \uloc(b_r\cC) \to \uloc(f_r\cC)$ is an equivalence for all $r$.
 We do this by induction on $r$, the case $r=0$ being trivial.
 
 Let $\grad_{r+1} \colon f_{r+1}\cC \to \cC$ be the composition of $\grad \colon f_{r+1}\cC \to \bigoplus \cC$ with the projection onto the $(r+1)$--th component.
 Then $\grad_{r+1} \circ \iota_{r+1} \simeq \id_\cC$, and there exists a natural transformation $\id \to \iota_{r+1} \circ \grad_{r+1}$.
 Since $\grad_{r+1}$ vanishes on $f_r\cC$, there is an induced functor $\overline{\grad}_{r+1} \colon f_{r+1}\cC/f_r\cC \to \cC$.
 On the quotient, the natural transformation $\id \to \iota_{r+1} \circ \grad_{r+1}$ becomes an equivalence since its fiber is contained in $f_r\cC$.
 Hence, $\overline{\grad_{r+1}} \colon f_{r+1}\cC/f_r\cC \to \cC$ is an equivalence.
 
 Consider the functor $\grad_{r+1} \circ \top \colon b_{r+1}\cC \to \cC$ next.
 Then $\Delta \circ \iota_{r+1}$ defines a functor in the opposite direction which satisfies $\grad_{r+1} \circ \top \circ \Delta \circ \iota_{r+1} \simeq \id$.
 There also exists a natural transformation $\id \to \Delta \circ \iota_{r+1} \circ \grad_{r+1} \circ \top$ since $\grad_{r+1} \circ \top \simeq \grad_{r+1} \circ \bot$,
 which induces a natural equivalence on $b_{r+1}\cC/b_r\cC$ by similar reasoning to the above.
  
 We thus have a commutative diagram
 \[\begin{tikzcd}
   b_{r+1}\cC/b_r\cC\ar[rr, "\top"]\ar[dr, "\overline{\grad}_{r+1} \circ \top"'] & & f_{r+1}\cC/f_r\cC\ar[dl, "\overline{\grad}_{r+1}"] \\
   & \cC &
 \end{tikzcd}\]
 Since $\overline{\grad}_{r+1} \circ \top$ and $\overline{\grad}_{r+1}$ are equivalences, $\top$ is also an equivalence.
  
 In particular, we obtain a map of split cofiber sequences
 \[\begin{tikzcd}
    \uloc(b_r\cC)\ar{r}\ar{d}{\uloc(\top)} & \uloc(b_{r+1}\cC)\ar{r}\ar{d}{\uloc(\top)} & \uloc(\cC)\ar{d}{\id} \\
    \uloc(f_r\cC)\ar{r} & \uloc(f_{r+1}\cC)\ar{r} & \uloc(\cC)
 \end{tikzcd}\]
 By induction, we conclude that $\uloc(b_r\cC) \to \uloc(f_r\cC)$ is an equivalence for all $r$.
\end{proof}

\begin{cor}\label{cor:k-omega-c-fib-bc/bqc-fc/fqc}
 There exists a natural equivalence
 \[ \grayson\cC\simeq \Omega\fib( \uloc(B\cC/B^q\cC) \xrightarrow{\uloc(\top)} \uloc(F\cC/F^q\cC) ). \]
\end{cor}
\begin{proof}
 Consider the map of cofiber sequences
 \[\begin{tikzcd}
  \uloc(B^q\cC)\ar{r}\ar{d}{\uloc(\top)} & \uloc(B\cC)\ar{r}\ar{d}{\uloc(\top)} & \uloc(B\cC/B^q\cC)\ar{d}{\uloc(\top)} \\
  \uloc(F^q\cC)\ar{r} & \uloc(F\cC)\ar{r} & \uloc(F\cC/F^q\cC)
 \end{tikzcd}\]
 Since $\uloc(B\cC) \xrightarrow{\uloc(\top)} \uloc(F\cC)$ is an equivalence by \cref{lem:fc-bc}, we obtain a natural equivalence
 \[ \Omega\fib( \uloc(B\cC/B^q\cC) \xrightarrow{\uloc(\top)} \uloc(F\cC/F^q\cC) ) \simeq \fib( \uloc(B^q\cC) \xrightarrow{\uloc(\top)} \uloc(F^q\cC) ). \]
 The claim follows by combining this with \cref{lem:grayson-as-fiber}.
\end{proof}

\begin{defn}
 Define the shift $(-)[1] \colon b\cC \to b\cC$ as the functor induced by the functor $f\cC \to f\cC$ arising from the map of posets $\bbN \to \bbN, i \mapsto \max\{0,i-1\}$.
\end{defn}

The next lemma is straightforward.

\begin{lem}\label{lem:shifting}
 The cofiber of the canonical natural transformation $(-)[1] \to \id_{b\cC}$ takes values in the essential image of the diagonal functor $\Delta \colon f^q\cC \to b^q\cC$.
\end{lem}

Recall that $u \colon F\cC \to f\cC$ denotes the forgetful functor.

\begin{prop}\label{prop:bc/btc-c}
 The functor $\pi \circ u \circ \top \colon B\cC/B^t\cC \to \cC$ is an equivalence.
\end{prop}
\begin{proof}
 It suffices to prove that $\pi \circ \top \colon b\cC/b^t\cC \to \cC$ is an equivalence.
 The functor $\Delta \circ \iota \colon \cC \to b\cC/b^t\cC$ provides a right-inverse to $\pi \circ \top$.
 
 Note that $\grad\Delta\iota\pi(X^\top)$ is zero in all but the lowest degree.
 Therefore, the transformation $\tau_{\top} \colon \top \to \iota \circ \pi \circ \top$ and the zero transformation $0 \colon \bot \to \iota \circ \pi \circ \top$
 induce a natural transformation $(-)[1] \to \Delta \circ \iota \circ \pi \circ \top$.
 In the resulting zig-zag $\id \leftarrow (-)[1] \to \Delta \circ \iota \circ \pi \circ \top$, the left-hand transformation becomes an equivalence in $b\cC/b^t\cC$ by \cref{lem:shifting}. Since the cofiber of the right-hand transformation lies in $b^t\cC$, this transformation also becomes an equivalence in $b\cC/b^t\cC$.
 
 Since $\Delta \circ \iota$ is a right-inverse to $\pi \circ \top$, the claim follows.
\end{proof}
\begin{rem}\label{rem:bc/btc-fc/fqc-reversed}
	By interchanging the roles of $\top$ and $\bot$, \cref{prop:bc/btc-c} also shows that the functor $\pi\circ u\circ \bot \colon B\cC/B^b\cC \to \cC$ is an equivalence.
\end{rem}

\begin{cor}\label{prop:bc/btc-fc/fqc}
 The functor $\top \colon B\cC/B^t\cC \to F\cC/F^q\cC$ is an equivalence.
\end{cor}
\begin{proof}
 Since $\pi$ is an equivalence by \cref{lem:gillet-waldhausen} and $\pi \circ \top$ is an equivalence by \cref{prop:bc/btc-c}, the claim follows.
\end{proof}

\begin{cor}\label{cor:btc/bqc-fib-bc/bqc-fc/fqc}
 There exists a natural equivalence
 \[ \uloc(B^t\cC/B^q\cC) \simeq \fib( \uloc(B\cC/B^q\cC) \xrightarrow{\uloc(\top)} \uloc(F\cC/F^q\cC) ). \]
\end{cor}
\begin{proof}
 Consider the map of cofiber sequences
 \[\begin{tikzcd}[column sep=.75em]
  \uloc(B^t\cC/B^q\cC)\ar{r}\ar{d} & \uloc(B\cC/B^q\cC)\ar{r}\ar{d}{\id} & \uloc(B\cC/B^t\cC)\ar{d}{\uloc(\top)} \\
  \fib\left( \uloc(B\cC/B^q\cC) \xrightarrow{\uloc(\top)} \uloc(F\cC/F^q\cC) \right)\ar{r} & \uloc(B\cC/B^q\cC)\ar{r} & \uloc(F\cC/F^q\cC)
 \end{tikzcd}\]
 Since $\uloc(\top)\colon \uloc(B\cC/B^t\cC)\to \uloc(F\cC/F^q\cC)$ is an equivalence by \cref{prop:bc/btc-fc/fqc}, the claim follows.
\end{proof}

Combining \cref{cor:k-omega-c-fib-bc/bqc-fc/fqc} and \cref{cor:btc/bqc-fib-bc/bqc-fc/fqc} we obtain the following result.
	\begin{cor}
		\label{cor:k-omega-btc/bqc}
		There is a natural equivalence
		\[\grayson\cC\simeq\Omega\uloc(B^t\cC/B^q\cC). \]
	\end{cor}

\begin{lem}\label{lem:binary-complex-k0}
	Let $X$ be a binary complex, i.e.~an object of $B\cC$. Then
	\[ [X^\top_{0,k}] = [X^\bot_{0,k}] \in K_0(\cC) \]
	for all $k \in \bbN$.
\end{lem}
\begin{proof}
	This follows by an easy induction since $\grad(X^\top) \simeq \grad(X^\bot)$.
\end{proof}

\begin{defn}\label{defn:f-chi-c}
 Denote by $\cC^\chi$ the full subcategory of $\cC$ spanned by those objects $X$ satisfying $[X] = 0 \in K_0(\cC)$.
\end{defn}

By \cref{lem:binary-complex-k0}, the functor $\pi \circ \bot$ restricts to a functor $\bot^\chi \colon B^t\cC \to \cC^\chi$ which vanishes on $B^q\cC$.

\begin{prop}\label{prop:top-contractible-euler-zero}
 The functor $\bot^\chi \colon B^t\cC/B^q\cC \to \cC^\chi$ induced by $\pi \circ \bot$ is an equivalence.
\end{prop}

The proof of \cref{prop:top-contractible-euler-zero} relies on the following construction.

\begin{const}\label{const:mu-x}
 Let $X \in b\cC$ be a binary complex supported on $[0,r]$ satisfying $[\colim X^\top] = 0 \in K_0(\cC)$, and hence also $[\colim X^\bot] = 0 \in K_0(\cC)$ by \cref{lem:binary-complex-k0}.
 Let $k \geq r$.
 Choose objects $A$, $B$, $S \in \cC$ which fit into cofiber sequences
 \[ A \xrightarrow{a} X^\top_k \oplus S \xrightarrow{b} B \quad \text{and} \quad A \xrightarrow{a'} S \xrightarrow{b'} B. \]
 These exist by virtue of \cref{lem:hellers-criterion}.
 
 Define $C \in f\cC$ as the filtered object
 \[ 0 \to \dots \to 0 \to A \xrightarrow{a} X^\top_k \oplus S \to 0 \to 0 \to \dots, \]
 where $A$ sits in degree $k-1$ and $X^\top_k \oplus S$ occupies degree $k$.
 
 Define a second filtered object $Y^\top$ as
 \[ X^\top_0 \to \dots \to X^\top_{k-2} \to X^\top_{k-1} \oplus A \xrightarrow{X^\top(k-1 \leq k) \oplus a'} X^\top_k \oplus S \to 0 \to \dots \]
 We observe that $\grad(Y^\top)$ and $\grad(X^\bot \oplus C)$ are canonically equivalent.
 Hence, the given data combine to a new binary complex which we denote by $\mu(X)$.
 Note that $\mu(X)$ depends on the choice of $k$ and the cofiber sequences $A \to X_k \oplus S \to B$ and $A \to S \to B$.
 
 Moreover, there is a canonical morphism $m \colon X \to \mu(X)$ such that $m^\bot$ is given by the inclusion $X^\bot \to X^\bot \oplus C$,
 and such that $m^\top$ is given by the inclusion of a direct summand up to degree $k$.
\end{const}

\begin{proof}[Proof of \cref{prop:top-contractible-euler-zero}]
 It suffices to show that $\bot^\chi \colon b^t\cC/b^q\cC \to \cC^\chi$ is an equivalence.
 By taking vertical Verdier quotients in the commutative square of exact functors
 \[\begin{tikzcd}
  b^q\cC\ar{r}\ar{d} & b^b\cC\ar{d} \\
  b^t\cC\ar{r} & b\cC
 \end{tikzcd}\]
 we obtain an exact functor $i \colon b^t\cC/b^q\cC \to b\cC/b^b\cC$.
 This functor fits into the commutative square of exact functors
 \[\begin{tikzcd}
  b^t\cC/b^q\cC\ar{r}{i}\ar{d}[swap]{\bot^\chi} & b\cC/b^b\cC\ar{d}{\pi\circ\bot} \\
  \cC^\chi\ar{r}{j} & \cC
 \end{tikzcd}\]
 We make the following claims:
 \begin{enumerate}
  \item\label{it:top-contractible-euler-zero-3} $\bot^\chi$ is essentially surjective.
  \item\label{it:top-contractible-euler-zero-1} $i$ is fully faithful.
 \end{enumerate}
 Since these claims in conjunction with \cref{rem:bc/btc-fc/fqc-reversed} imply that $\bot^\chi$ is an equivalence, it suffices to prove the claims.
 
 Let us first show claim~\eqref{it:top-contractible-euler-zero-3}.
 For any $X \in \cC^\chi$, apply \cref{const:mu-x} to $\Delta(\iota(X))$ to obtain a preimage of $X$ under $\bot^\chi$.
  
 For claim~\eqref{it:top-contractible-euler-zero-1}, we rely on \cite[Theorem~I.3.3]{NikolausScholze} to reduce the claim to showing that the canonical map
 (induced by the inclusion of indexing categories)
 \[ \colim_{Z \to Y \in b^q\cC_{/Y}} \Map_{b^t\cC}(X, \cofib(Z \to Y)) \to \colim_{Z \to Y \in b^b\cC_{/Y}} \Map_{b\cC}(X, \cofib(Z \to Y)) \]
 is an equivalence.
 Since $b^t\cC \to b\cC$ is fully faithful, it suffices to show that the inclusion $b^q\cC_{/Y} \subseteq b^b\cC_{/Y}$ is cofinal.
 Using \cite[Theorem~4.1.3.1]{MR2522659}, this in turn can be reduced to showing that
 \[ \cF := b^q\cC_{/Y} \times_{b^b\cC_{/Y}} (b^b\cC_{/Y})_{(Z \to Y)/} \]
 is weakly contractible for all $Z \to Y \in b^b\cC_{/Y}$.
 In fact, we claim that $\cF$ is filtered.
 
 Let $K$ be a finite simplicial set, and suppose we are given a map $f \colon K \to \cF$.
 Then $f$ corresponds to a diagram $\Delta^0 \star K \to b^b\cC_{/Y}$ with the following properties:
 \begin{itemize}
  \item The restriction to $\Delta^0 \star \emptyset$ classifies the object $Z \to Y$.
  \item The restriction to $K$ factors via $b^q\cC_{/Y}$.
 \end{itemize}
 Since $b^b\cC$ admits finite colimits, there is a universal cone $\hat{f}' \colon (\Delta^0 \star K)^\rhd \to b^b\cC_{/Y}$.
 
 Let $W \in b^b\cC$ denote the object classified by the cone point of $\hat{f}'$.
 Choose $k \in \bbN$ such that both $Y$ and $W$ are supported on $[0,k]$.
 Now apply \cref{const:mu-x} to $W$ to obtain $\mu(W)$.
 We claim that the dashed arrow in the following diagram can be filled in such that the resulting triangle commutes:
 \[\begin{tikzcd}
    W\ar{r}{m}\ar{dr} & \mu(W)\ar[dashrightarrow]{d} \\
     & Y
   \end{tikzcd}\]
 To construct the required $2$--simplex in $b\cC$, we rely on \cref{rem:bc-explicit-description}.
 
 For any filtered object $V$, let $V^{\leq k}$ denote the restriction of $V$ along the inclusion $[0,k] \to \bbN$.
 By construction, we find commutative diagrams
 \[\begin{tikzcd}
    \topbot(W)^{\leq k}\ar[r, "m"]\ar[dr] & \topbot(\mu(W))^{\leq k}\ar[d] \\
     & \topbot(Y)^{\leq k}
   \end{tikzcd}\]
 for $\topbot \in \{\top,\bot\}$ since $\topbot(W)^{\leq k}$ is a direct summand in $\topbot(\mu(W))^{\leq k}$.
 These diagrams extend essentially uniquely to commutative diagrams in $f\cC$ since $\topbot(Y)(i) \simeq 0$ for $i > k$.
 By definition, the chosen equivalences of graded objects patch together to yield a $2$--simplex in $b\cC$.
 
 It follows that we can replace the cone point $W$ by $\mu(W)$ and still obtain a diagram $\hat{f} \colon (\Delta^0 \star K)^\rhd \to b^b\cC_{/Y}$.
 Since $\mu(W) \in b^q\cC$, this corresponds to a diagram $K^\rhd \to \cF$, so $\cF$ is filtered.
   
 This proves that the inclusion $b^q\cC_{/Y} \subseteq b^b\cC_{/Y}$ is cofinal;
 thus, we have shown claim~\eqref{it:top-contractible-euler-zero-1}.
\end{proof}

\begin{proof}[Proof of \cref{thm:mainintro}]
 Combining \cref{cor:k-omega-btc/bqc} and \cref{prop:top-contractible-euler-zero},
 we obtain the following sequence of natural equivalences:
 \[ \grayson\cC \simeq \Omega\uloc(B^t\cC/B^q\cC) \simeq \Omega\uloc(\cC^\chi).\]
 Since every object in $\cC$ is a retract of an object in $\cC^\chi$, it follows that $\uloc(\cC^\chi) \simeq \uloc(\cC)$.
 This proves the theorem.
\end{proof}

Since the diagonal $F^q \to B^q$ defines a natural transformation of endofunctors on $\catex$, the construction $\grayson\cC$ can be iterated.
Any word $\bfW$ of length $n$ over the alphabet $\{F^q, B^q\}$ defines a stable $\infty$--category $\bfW\cC$.
Letting the word $\bfW$ vary, the $2^n$ possible choices of $\bfW$ assemble into a commutative cube of dimension $n$.
By taking the $n$--fold iterated cofiber of this cube, we obtain a motive $\grayson^n\cC$.
To make sense of the naturality of the assignment $\cC \mapsto \grayson^n\cC$, we choose, once and for all, one preferred order of taking cofibers,
say reading the word $\bfW$ from left to right.

\begin{cor}\label{cor:iterated-grayson}
  For all $n \geq 1$ there is a natural equivalence
 \[ \grayson^n\cC \simeq \Omega^n\uloc(\cC). \]
\end{cor}
\begin{proof}
 The proof is by induction, with the case $n=1$ being \cref{thm:mainintro}.
 By induction hypothesis, and using \cref{thm:mainintro} again, we have
 \begin{align*}
  \grayson^{n+1}\cC
  &\simeq \cofib( \grayson^nF^q\cC \xrightarrow{\grayson^n\Delta} \grayson^nB^q\cC) \\
  &\simeq \cofib( \Omega^n\uloc(F^q\cC) \xrightarrow{\uloc(\Delta)} \Omega^n\uloc(B^q\cC) ) \\
  &\simeq \Omega^n \grayson\cC \\
  &\simeq \Omega^{n+1}\uloc(\cC).\qedhere
 \end{align*} 
\end{proof}

\begin{cor}\label{cor:iterated-grayson-K}
 For all $n \geq 1$ there are natural equivalences
 \[ \bfK(\grayson^n\cC) \simeq \Omega^n\bfK(\cC) \quad \text{and} \quad \Kcon(\grayson^n\cC) \simeq \Omega^n\tau_{\geq n}\Kcon(\cC). \]
\end{cor}
\begin{proof}
 Using \cref{cor:iterated-grayson}, the statement about non-connective $K$--theory is immediate because $\bfK$ is a localizing invariant.
 The claim about connective $K$--theory follows by taking connective covers.
\end{proof}

From \cref{cor:iterated-grayson-K}, we can now deduce a more algebraic description of higher $K$--groups resembling \cite[Corollary~7.4]{Grayson2012}.

\begin{prop}\label{prop:k_n-b-generators-and-relations}
 Let $\cC$ be an idempotent complete stable $\infty$--category and let $n \geq 1$.
 Then $\pi_n\bfK(\cC) \cong \pi_0\bfK(\grayson^n\cC)$ admits the following presentation.
 
 It is the abelian group generated by equivalence classes $[X]$ of objects in $(B^q)^n\cC$, subject to the relations
 \begin{enumerate} 
  \item $[X] = [X'] + [X'']$ whenever there exists a cofiber sequence $X' \to X \to X''$ in $(B^q)^n\cC$;
  \item $[X] = 0$ if $X$ lies in the essential image of some diagonal functor
   \[ (B^q)^kF^q(B^q)^{n-k-1}\cC \xrightarrow{\Delta} (B^q)^n\cC,\qquad 0 \leq k \leq n-1. \]
 \end{enumerate}
 By abuse of notation, we will refer to the above group also by $K_0(\grayson^n\cC)$.
\end{prop}
\begin{proof}
 Consider first the case $n=1$.
 Since idempotent completeness can be characterized by the existence of certain colimits in $\cC$ (cf.~\cite[Section~4.4.5]{MR2522659}),
 it follows from \cite[Corollary~5.1.2.3]{MR2522659} that $F^q\cC$ is idempotent complete.
 As $\grad \colon F^q\cC \to \bigoplus_\bbN \cC$ preserves colimits,
 $B^q\cC$ is idempotent complete by \cite[Lemma~5.4.5.5]{MR2522659}.
 This allows us to identify $K_0(F^q\cC) \cong \pi_0\bfK(F^q\cC)$ and $K_0(B^q\cC) \cong \pi_0\bfK(B^q\cC)$.
 
 Since the diagonal functor $\Delta \colon F^q\cC \to B^q\cC$ admits a retraction (by $\top$ or $\bot$), we obtain a split exact sequence of abelian groups
 \[ 0 \to K_0(F^q\cC) \to K_0(B^q\cC) \to \pi_0\bfK(\grayson\cC) \to 0. \]
 By \cref{cor:iterated-grayson}, we have $\pi_1\bfK(\cC) \cong \pi_0\bfK(\grayson\cC)$, and the claim follows immediately.
 
 The general case follows by considering cubes of higher dimensions.
\end{proof}

\section{Shortening binary complexes}\label{sec:shortening}
Our next goal is to show that the explicit description of $K_n$ in \cref{prop:k_n-b-generators-and-relations} includes a large number of superfluous generators and relations.
To make this statement precise, we introduce the following variations of the categories $F^q\cC$, $B^q\cC$ and the motive $\grayson\cC$. In this section, we always assume that $\cC$ is idempotent complete.

\begin{defn}
 Let $r \geq 0$. Denote by $F^q_r\cC \subseteq F^q\cC$ the full subcategory of bounded filtrations supported on $[0,r]$.
 Define $B^q_r\cC$ as the pullback of the diagram
 \[ F^q_r\cC \xrightarrow{\grad} \bigoplus_{i = 0}^r \cC \xleftarrow{\grad} F^q_r\cC, \]
 and set
 \[ \grayson_r\cC := \cofib( \uloc(F^q_r\cC) \xrightarrow{\uloc(\Delta)} \uloc(B^q_r\cC)) \in \mloc. \]
\end{defn}

The evident inclusion functors induce a map $\grayson_r\cC \to \grayson\cC$.
The key result of this section, \cref{prop:split} below, states that the induced homomorphism
$K_0(\grayson_r\cC) \to K_0(\grayson\cC)$ admits a natural section for $r = 7$.

For convenience, we regard $\grayson_r\cC$ as the cofiber of the map $\uloc(f^q_r\cC) \xrightarrow{\uloc(\Delta)} \uloc(b^q_r\cC)$.

Note that the complexes in $f\cC$ start with a zero object in degree 0. In this section we will suppress this zero and write complexes starting with degree 1.

Let $X \in b^q\cC$. By \cref{lem:hellers-criterion,lem:binary-complex-k0} we can choose objects $A_k,B_k,S_k$
fitting into cofiber sequences $A_k\to X^\top_k\oplus S_k\to B_k$ and $A_k \to X^\bot_k \oplus S_k \to B_k$ for each $k \geq 3$.
If $X^\top_k$ and $X^\bot_k$ are both trivial, also choose $A_k, B_k$ and $S_k$ to be trivial. 
Let $f_k^\top \colon B_k \to \Sigma A_k$ and $f_k^\bot \colon B_k\to \Sigma A_k$ be the induced maps.
Define $Y_2$ to be the complex with
\[ Y_2^\top := \begin{tikzcd}
		X^\top_1 \ar[r] & X^\top_2\ar[r] & X^\top_3\ar[d, phantom, "\oplus"] \ar[r] & 0 \ar[r] & 0 \ar[r] & \dots \\
		0 \ar[r] & 0 \ar[r] & B_3\ar[r, "f_3^\bot"] & \Sigma A_3 \ar[r] & 0 \ar[r] & \dots
		\end{tikzcd} \]
and
\[ Y_2^\bot := \begin{tikzcd}
		X^\bot_1 \ar[r] & X^\bot_2\ar[r] & X^\bot_3\ar[d, phantom, "\oplus"] \ar[r] & 0 \ar[r] & 0 \ar[r] & \dots \\
		0 \ar[r] & 0 \ar[r] & B_3\ar[r, "f_3^\top"] & \Sigma A_3 \ar[r] & 0 \ar[r] & \dots
		\end{tikzcd} \]
By construction, there is a canonical equivalence $\grad(Y^\top_2) \simeq \grad(Y^\bot_2)$.

Similarly, define $Y_k$ for $k \geq 3$ to be the complex with
\[ Y_k^\top := \begin{tikzcd}
		\Sigma^{-1} B_k\ar[r, "\Sigma^{-1}f_k^\bot"] & A_k \ar[r] & 0\ar[d, phantom, "\oplus"] \ar[r] & 0\ar[r] & 0\ar[r] & \dots \\
		0\ar[r] & X^\top_k \ar[r] & X^\top_{k+1}\ar[r]\ar[d, phantom, "\oplus"] & 0 \ar[r] & 0 \ar[r] & \dots \\
		0 \ar[r] & 0 \ar[r] & B_{k+1}\ar[r, "f_{k+1}^\bot"] & \Sigma A_{k+1} \ar[r] & 0 \ar[r] & \dots
		\end{tikzcd} \]
and
\[ Y_k^\bot := \begin{tikzcd}
		\Sigma^{-1} B_k\ar[r, "\Sigma^{-1}f_k^\top"] & A_k \ar[r] & 0\ar[d, phantom, "\oplus"] \ar[r] & 0\ar[r] & 0\ar[r] & \dots \\
		0\ar[r] & X^\bot_k \ar[r] & X^\bot_{k+1}\ar[r]\ar[d, phantom, "\oplus"] & 0 \ar[r] & 0 \ar[r] & \dots \\
		0 \ar[r] & 0 \ar[r] & B_{k+1}\ar[r, "f_{k+1}^\top"] & \Sigma A_{k+1} \ar[r] & 0 \ar[r] & \dots
		\end{tikzcd} \]
It follows again from the choice of $f_k^\top$ and $f_k^\bot$ that we have canonical equivalences $\grad(\top(Y_k)) \simeq \grad(\bot(Y_k))$. Note that $Y_k$ is an object of $b_5^q\cC$.
\begin{prop}
	\label{prop:surj}
	In $K_0(\grayson\cC)$ we have
	\[[X]=\sum_{k\geq 2}[Y_k].\]
	In particular, $K_0(\grayson_5\cC) \to K_0(\grayson\cC)$ is surjective.
\end{prop}
\begin{proof}
	We denote the morphisms $X_k^\top\to X_{k+1}^\top$ by $x^\top_k$.
	Let $X^\top_{k+1}/X^\top_k:=\cofib(x_k^\top)$ and $X^\bot_{k+1}/X^\bot_k := \cofib(x_k^\bot)$.
	We have a morphism $X[1]\to Y_2$ with top component given as follows and bottom component given analogously. The induced map on the underlying graded is trivial.
	\[\begin{tikzcd}	0\ar[r]\ar[d]&X^\top_1\ar[d,"x_1^\top"]\ar[r,"x_1^\top"]&X^\top_2\ar[d,"{(x_2^\top,0)}"]\ar[r,"x_2^\top"]&X^\top_3\ar[d,"0"]\ar[r,"x_3^\top"]&X^\top_4\ar[d,"0"]\ar[r,"x_4^\top"]&\ldots\\
	X_1^\top\ar[r,"x_1^\top"]&X_2^\top\ar[r,"{(x_2^\top,0)}"]&X_3^\top\oplus B_3\ar[r,"0+f_3^\bot"]&\Sigma A_3\ar[r]&0\ar[r]&\ldots\end{tikzcd}\]
	By inspection, we see that the fiber has top component
	\[\Sigma^{-1}X_1^\top \xrightarrow{0} \Sigma^{-1} X^\top_2/X^\top_1\xrightarrow{0}\Sigma^{-1}X^\top_3/X^\top_2\oplus\Sigma^{-1}B_3 \xrightarrow{0\oplus \Sigma^{-1} f^\bot_3} X^\top_3 \oplus A_3 \xrightarrow{x^\top_3 + 0} X^\top_4 \to \dots \]
	and bottom component
	\[\Sigma^{-1}X_1^\bot \xrightarrow{0} \Sigma^{-1} X^\bot_2/X^\bot_1\xrightarrow{0}\Sigma^{-1}X^\bot_3/X^\bot_2\oplus\Sigma^{-1}B_3 \xrightarrow{0\oplus \Sigma^{-1} f^\top_3} X^\bot_3 \oplus A_3 \xrightarrow{x^\bot_3 + 0} X^\bot_4 \to \dots \]
	Note that the fiber splits as the direct sum of the diagonal of the complex
	\[\Sigma^{-1}X_1^\top \xrightarrow{0} \Sigma^{-1} X^\top_2/X^\top_1\xrightarrow{0}\Sigma^{-1}X^\top_3/X^\top_2\to 0\to\ldots\]
	and the complex $F_1$ with top component
	\[ 0 \to 0 \to \Sigma^{-1}B_3 \xrightarrow{(0, \Sigma^{-1} f^\bot_3)} X^\top_3 \oplus A_3 \xrightarrow{x_3^\top + 0} X^\top_4 \to X^\top_5 \to \dots \]
	and bottom component
	\[ 0 \to 0 \to \Sigma^{-1}B_3 \xrightarrow{(0, \Sigma^{-1} f^\top_3)} X^\bot_3 \oplus A_3 \xrightarrow{x_3^\bot + 0} X^\bot_4 \to X^\bot_5 \to \dots \]
	We conclude that $[X[1]] = [F_1] + [Y_2] \in K_0(\grayson\cC)$.
	By \cref{lem:shifting}, we have $[X]=[X[1]]$ and we can shift $F_1$ down. The claim now follows by induction on the length of the support of $X$.
\end{proof}

\begin{prop}
	\label{prop:split}
	Sending $[X]$ to $\sum_{k\geq 2}[Y_k]$ yields a natural splitting of the natural map
	\[K_0(\grayson_7\cC) \to K_0(\grayson\cC).\]
\end{prop}
\begin{proof}
	We first show that $\sum_{k\geq 2}[Y_k]$ is independent of the choices of $A_k, B_k$ and $S_k$. Let $k\geq 3$ be given.
	
%
%
%
	Consider the following cofiber sequence of bounded filtrations supported on $[0,7]$ which is obtained by considering the right hand map and taking its fiber:
	\begin{equation}\begin{tikzcd}[ampersand replacement=\&, column sep=1.9em]\label{eq:split-well-defined}
	 \Sigma^{-2}B_k\ar[r]\ar[d,"0"] \& 0 \ar[d]\ar[r] \& \Sigma^{-1}B_{k}\ar[d, "{(0, \Sigma^{-1} f_{k}^\bot)}"] \\
	 \Sigma^{-1}X_k^\top\oplus \Sigma^{-1}X_k^\bot\oplus\Sigma^{-1}S\ar[r]\ar[d,"0"] \& \Sigma^{-1}B_k \ar[d, "{(0,\Sigma^{-1}f_k^\bot)}"]\ar[r,"{(0,\Sigma^{-1}f_k^\bot)}"] \& X_{k}^\top \oplus A_{k}\ar[d, "x_{k}^\top \oplus 0"] \\
	 \Sigma^{-1}X_{k+1}^\top/X_{k}^\top\oplus A_k\oplus \Sigma^{-1}B_{k+1}\ar[d,"{(0, 0+0+\Sigma^{-1} f_{k+1}^\bot)}"]\ar[r] \& X_k^\top\oplus A_k\ar[r,"x_k^\top+0"]\ar[d, "x_k^\top+0"] \& X_{k+1}^\top \oplus B_{k+1}\ar[d,"0+f_{k+1}^\bot"]  \\
	 X_{k+1}^\top \oplus A_{k+1}\ar[d, "x_{k+1}^\top \oplus 0"]\ar[r] \& X_{k+1}^\top \ar[r,"0"]\ar[d, "x_{k+1}^\top+0"] \& \Sigma A_{k+1} \ar[d] \\
	 X_{k+2}^\top \oplus B_{k+2}\ar[d, "0 + f_{k+2}^\bot"]\ar[r] \& X_{k+2}^\top\oplus B_{k+2}\ar[r]\ar[d, "0+f_{k+2}^\bot"] \& 0\ar[d]\\
	 \Sigma A_{k+2}\ar[r]\ar[d] \& \Sigma A_{k+2}\ar[r]\ar[d] \& 0\ar[d] \\
	 0\ar[r] \& 0\ar[r] \& 0
	\end{tikzcd}\end{equation}
	Note that the 2-cells on the right are all trivial, but this is not true for the 2-cells on the left.
	Interchanging $\top$ and $\bot$ in \eqref{eq:split-well-defined} defines a second cofiber sequence of bounded filtrations supported on $[0,7]$. Note that the two right hand maps define a morphism in $b^q_7\cC$ where the map on the underlying graded objects is trivial. Taking the fiber of this map yields an object $F_k$ in $b_7^q\cC$ whose top component is given by the left hand column in \eqref{eq:split-well-defined} and whose bottom component is given analogously. 
	
	Note that the right-hand column in the two versions of \eqref{eq:split-well-defined} is $Y_k$,
	while the fiber $F_k$ defines the sum of a shifted copy of $Y_{k+1}$ and a binary complex contained in the essential image of the diagonal functor.
	Using \cref{lem:shifting}, it follows that $[Y_k] + [Y_{k+1}] \in K_0(\grayson_7\cC)$ is equal to the class of the binary acyclic complex defined by the middle column.
	Now it is enough to observe that the middle column is independent of the choice of $A_{k+1}, B_{k+1}$ and $S_{k+1}$.
		
	The independence of $A_3$, $B_3$ and $S_3$ follows in the same way.
	
	This shows that the assignment $X \mapsto \sum_{k\geq 2} [Y_k]$ gives a well-defined map
	\begin{equation}\label{eq:split-as-sets} \operatorname{ob} b^q\cC \to K_0(\grayson_7\cC), \end{equation}
	which evidently sends equivalent objects to the same class.
	
	Let $X\to X'\to X''$ be a cofiber sequence in $b^q\cC$.
	Consider the cofiber sequences $X_k^\top\to (X')_k^\top \to (X'')_k^\top$ and $X_k^\bot \to (X')_k^\bot \to (X'')_k^\bot$, $k \in \bbN$, as elements in $K_0(\cE\cC)$. Since $[X_k^\top] = [X_k^\bot]$ and $[(X'')_k^\top] = [(X'')_k^\bot]$ in $K_0(\cC)$ by \cref{lem:binary-complex-k0}, we conclude from \cref{ex:classical-additivity} that
	\[ [X_k^\top\to (X')_k^\top \to (X'')_k^\top] = [X_k^\bot \to (X')_k^\bot \to (X'')_k^\bot] \in K_0(\cE\cC). \]
	From \cref{lem:hellers-criterion}, it follows that there are cofiber sequences $A_k\to A_k'\to A_k''$, $B_k\to B_k'\to B_k''$ and $S_k\to S_k'\to S_k''$ fitting into cofiber sequences of cofiber sequences as follows:
	\[\begin{tikzcd}
	A_k\ar [r]\ar [d]&X_k^\top\oplus S_k\ar [r]\ar [d]& B_k \ar [d]\\
	A_k'\ar [r]\ar [d]&(X')^\top_k\oplus S_k'\ar [r]\ar [d]& B_k' \ar [d]\\
	A_k''\ar [r]&(X'')_k^\top\oplus S_k''\ar [r]& B_k''
	\end{tikzcd}
	\quad\text{and}\quad
	\begin{tikzcd}
	A_k\ar [r]\ar [d]&X_k^\bot\oplus S_k\ar [r]\ar [d]& B_k \ar [d]\\
	A_k'\ar [r]\ar [d]&(X')_k^\bot\oplus S_k'\ar [r]\ar [d]& B_k' \ar [d]\\
	A_k''\ar [r]&(X_k'')_k^\bot\oplus S_k''\ar [r]& B_k''
	\end{tikzcd}
	\]
	Using these for the construction of the objects $Y_k, Y_k', Y_k''$ of $b^q_5$ as above, we obtain cofiber sequences $Y_k\to Y_k'\to Y_k''$.
	Hence \eqref{eq:split-as-sets} induces a well-defined homomorphism $K_0(b^q\cC)\to K_0(\grayson_7\cC)$.
	
	If $X^\top \simeq X^\bot$, we can make our choices such that $f_k^\top\simeq f_k^\bot$ for all $k\geq 3$ and hence $Y_k^\top \simeq Y_k^\bot$ for all $k\geq 2$. Therefore, the map induces a well-defined homomorphism $K_0(\grayson\cC) \to K_0(\grayson_7\cC)$.
	It is a split of the natural map $K_0(\grayson_7\cC) \to K_0(\grayson\cC)$ by \cref{prop:surj}.
\end{proof}

The preceding \cref{prop:surj,prop:split} prove \cref{thm:split2}. We obtain the following generalization to higher algebraic $K$-theory.

\begin{thm}
	\label{thm:split}
	The canonical map $K_0(\grayson^n_5\cC)\to K_0(\grayson^n\cC)$ is a surjection and the canonical map $K_0(\grayson_7^n\cC) \to K_0(\grayson^n\cC) \cong K_n(\cC)$ admits a natural section.
\end{thm}
\begin{proof}
	We argue by induction.
	\cref{prop:surj} and \cref{prop:split} prove the case $n=1$, which is the start of the induction.
	The induction step is analogous to \cite[Remark~8.1]{Grayson2012}, \cite[Proof of Theorem~1.4]{KasprowskiWinges} and \cref{cor:iterated-grayson} above.
	
	We will only describe the induction for the existence of the natural section, the induction for the surjectivity statement is completely analogous.
	
	In order to extend \cref{prop:split} from $K_1$ to higher $K$--groups, we require the additional observation that $\Gamma$ and $\Gamma_r$ commute.
	This can be morally seen by permuting the two factors in $\bbN\times\bbN$, but the formal argument is rather lengthy.
	Hence we will first complete the proof using this claim before giving the formal argument.
	
	The map $K_0(\grayson^n_7\cC)\to K_0(\grayson\grayson_7^{n-1}\cC)$ admits a natural section by \cref{prop:split}
	because it is a natural retract of the homomorphism $K_0(\grayson_7 B_7^{n-1}\cC) \to K_0(\grayson B_7^{n-1}\cC)$.
	Since $\grayson$ and $\grayson_7$ commute, it suffices to show that $K_0(\grayson_7^{n-1}\grayson\cC)\to K_0(\grayson^n\cC)$ admits a natural section.
	Since this map is in turn a natural retract of the map $K_0(\grayson_7^{n-1}B^q\cC)\to K_0(\grayson^{n-1}B^q\cC)$,
	this follows from the induction assumption.
	
	What is left to do is to provide an argument why $\grayson$ and $\grayson_7$ may be permuted.
	Fix $r \in \bbN$. Let $\bfW$ be a word of length $n$ over the alphabet $\{B^q, B^q_r\}$, and let $\sigma \in S_n$ be a permutation.
	We claim that there is a canonical equivalence $\bfW\cC \simeq \bfW_\sigma\cC$, where $\bfW_\sigma$ denotes the word $\bfW$ permuted according to $\sigma$.
	
	Recall that the natural transformation $\grad \colon F^q \to \bigoplus_\bbN \cC$ is obtained via pullback with a map of posets $\gamma \colon \bbN^\disc \to \gapindex{\bbN}$.
	Letting $\bbN(0) := \gapindex{\bbN}$ and $\bbN(1) := \bbN^\disc$, define for $x = (x_1,\dots,x_n) \in \{0,1\}^n$
	\[ \bbN(x) := \prod_{i=1}^n \bbN(x_i). \]
	Consider the functor
	\[ \bfN \colon [1]^n \to \{\text{posets}\}, x \mapsto \bbN(x) \]
	induced by $\gamma$.
	Then $\sigma$ induces a natural isomorphism $\bfN \xrightarrow{\cong} \sigma^*\bfN$ to the diagram of posets obtained by permuting the coordinates according to $\sigma$.
	
	Let $\Fun(\bfN,\cC) \colon (\Delta^1)^n \to \catex$ denote the induced $n$--cube of stable $\infty$--categories.
	Then we obtain an induced equivalence of functors
	\[ \Fun(\sigma^*\bfN,\cC) \xrightarrow{\sim} \Fun(\bfN,\cC), \]
	which contains as a full subfunctor those cubes in which we restrict to $F^q$ and $F^q_r$ at the appropriate places (according to the original choice of word $\bfW$).
	
	Let $\widetilde{\Fun}(\bfN,\cC) \colon (\sd\Delta^1)^n \to \catex$ denote the functor obtained from $\Fun(\bfN,\cC)$
	by further precomposing with the map induced by the map of simplicial sets $\sd\Delta^1 \to \Delta^1$
	which sends the endpoints of the subdivided $1$--simplex to $0$ and the subdivision point to $1$.
	
	Then $\bfW\cC$ is the limit of $\widetilde{\Fun}(\bfN,\cC)$, while $\bfW_\sigma\cC$ can be obtained as the limit of $\widetilde{\Fun}(\sigma^*\bfN,\cC)$.
	Since we have seen that the diagrams $\widetilde{\Fun}(\bfN,\cC)$ and $\widetilde{\Fun}(\sigma^*\bfN,\cC)$ are equivalent, we obtain the desired equivalence
	\[ \bfW\cC \simeq \bfW_\sigma\cC. \]
	Let $\bfW'$ denote the word over the alphabet $\{F^q, F^q_r\}$ obtained from $\bfW$ by replacing $B$ with $F$.
	Then there is an evident diagonal transformation $\bfW'\cC \xrightarrow{\Delta} \bfW\cC$ which fits into a commutative square
	\[\begin{tikzcd}
	   \bfW'\cC\ar[r, "\Delta"]\ar[d, "\simeq"] & \bfW\cC\ar[d, "\simeq"] \\
	   \bfW_\sigma'\cC\ar[r, "\Delta"] & \bfW_\sigma\cC
	  \end{tikzcd}\]
	Applying $\uloc$ and taking horizontal cofibers proves that
	\[ \bfV\cC \simeq \bfV_\sigma\cC \]
	for any word $\bfV$ over the alphabet $\{\grayson, \grayson_r\}$.
	In particular, $\grayson$ and $\grayson_7$ commute.
\end{proof}

\section{Infinite products}\label{sec:products}
This section is devoted to the proof of \cref{thm:commute}. The proof of \cref{thm:commute} for connective $K$-theory is almost verbatim the same as the proof of \cite[Theorem~4.1]{KasprowskiWinges}.

\begin{lem}\label{lem:k_0-products}
 The functor $K_0 \colon \catex \to \Ab$ commutes with infinite products.
\end{lem}
\begin{proof}
 Let $\{\cC_i\}_{i \in I}$ be a family of stable $\infty$--categories.
 The natural comparison map $K_0(\prod_{i \in I} \cC_i) \to \prod_{i \in I} K_0(\cC_i)$ is obviously surjective.
 Injectivity follows from \cref{lem:hellers-criterion}.
\end{proof}

The next lemma, whose proof is similar to the argument in the proof of \cite[Theorem~1.2]{KasprowskiWinges},
shows that Verdier quotients are compatible with the formation of products.

\begin{lem}\label{lem:products-exact-sequences}
 Let $\{\cD_i \to \cC_i \to \cC_i/\cD_i \}_{i \in I}$ be a family of Verdier sequences in $\catex$.
 
 Then there is a natural equivalence
 \[ \prod_{i \in I} \cC_i / \prod_{i \in I} \cD_i \simeq \prod_{i \in I} \cC_i/\cD_i. \]
\end{lem}
\begin{proof}
 Consider the commutative diagram of stable $\infty$--categories and exact functors
 \[\begin{tikzcd}[column sep=small]
    & \prod_{i \in I} \cC_i\ar[rd, "\ell"]\ar[ld, "\ell'"'] & \\
    \prod_{i \in I} \cC_i / \prod_{i \in I} \cD_i\ar[rr, "f"] & & \prod_{i \in I} \cC_i/\cD_i
   \end{tikzcd}\]
 Since the localization functor $\ell$ is essentially surjective, so is $f$.
 Therefore, it suffices to show that $f$ is fully faithful.
 Using \cite[Theorem~I.3.3]{NikolausScholze} to compute mapping spaces in the localization and referring to \cite[Lemma~3.10]{BSS} for the fact that filtered colimits distribute over products in spaces, we conclude that
 \begin{align*}
  \Map&_{\prod_{i \in I} \cC_i / \prod_{i \in I} \cD_i}((X_i)_i, (Y_i)_i) \\
  &\simeq \colim_{((Z_i)_i \to (Y_i)_i \in (\prod_{i \in I} \cD_i)_{/(Y_i)_i}} \Map_{\prod_{i \in I} \cC_i}((X_i)_i, \cofib((Z_i)_i \to (Y_i)_i)) \\
  &\simeq \colim_{((Z_i)_i \to (Y_i)_i \in (\prod_{i \in I} \cD_i)_{/(Y_i)_i}} \prod_{i \in I} \Map_{\cC_i}(X_i, \cofib(Z_i \to Y_i)) \\
  &\simeq \prod_{i \in I} \colim_{Z_i \to Y_i \in (\cD_i)_{/Y_i}} \Map_{\cC_i}(X_i, \cofib(Z_i \to Y_i)) \\
  &\simeq \prod_{i \in I} \Map_{\cC_i/\cD_i}(X_i, Y_i) \\
  &\simeq \Map_{\prod_{i \in I} \cC_i/\cD_i}(X_i, Y_i),
 \end{align*}
 so $f$ is also fully faithful.
\end{proof}

\begin{lem}
	\label{lem:idem}
	Let $\{\cC_i\}_{i \in I}$ be a family of stable $\infty$--categories. 
	The canonical functor
	\[\Idem(\prod_{i\in I}\cC_i)\to \prod_{i\in I}\Idem(\cC_i)\]
	is an equivalence.	
\end{lem}
\begin{proof}
 The canonical functor $\prod_{i \in I} \cC_i \to \prod_{i \in I} \Idem(\cC_i)$ exhibits $\prod_{i \in I} \Idem(\cC_i)$ as an idempotent completion of $\prod_{i \in I} \cC_i$ in the sense of \cite[Definition~5.1.4.1]{MR2522659}: Since idempotent completeness amounts to the existence of certain colimits (\cite[Section~4.4.5]{MR2522659}) and colimits in a product category can be computed componentwise, $\prod_{i \in I} \Idem(\cC_i)$ is idempotent complete; moreover, every object in $\prod_{i \in I} \Idem(\cC_i)$ is a retract of an object in $\prod_{i \in I} \cC_i$ because this is true for each individual component. 	
\end{proof}

\begin{prop}\label{prop:commute-pi_n}
 Let $\{\cC_i\}_{i \in I}$ be a family of stable $\infty$--categories.
 The comparison map
 \[ \pi_n\bfK(\prod_{i \in I} \cC_i) \to \prod_{i \in I} \pi_n\bfK(\cC_i) \]
 is an isomorphism for all $n \in \bbZ$.
\end{prop}
\begin{proof}
	By \cref{lem:idem} and the fact that $\bfK$ is a localizing invariant, we may assume without loss of generality that all $\cC_i$ are idempotent complete.
 
 The case $n = 0$ is provided by \cref{lem:k_0-products}.
 For $n \geq 1$, we consider the commutative diagram
 \[\begin{tikzcd}
    K_0(\grayson^n\prod_{i \in I} \cC_i)\ar[r]\ar[d] & \prod_{i \in I} K_0(\grayson^n\cC_i)\ar[d] \\
    K_0(\grayson^n_7\prod_{i \in I} \cC_i)\ar[r, "\phi"]\ar[d] & \prod_{i \in I} K_0(\grayson^n_7\cC_i)\ar[d] \\
    K_0(\grayson^n\prod_{i \in I} \cC_i)\ar[r] & \prod_{i \in I} K_0(\grayson^n\cC_i)
   \end{tikzcd}\]
 in which the vertical homomorphisms are given by the section of \cref{thm:split} followed by the homomorphism induced by the canonical maps $\grayson^n_7\cC_i \to \grayson^n\cC_i$.
 In particular, \cref{thm:split} together with \cref{cor:iterated-grayson-K} tells us that this diagram exhibits the comparison map $\pi_n\bfK(\prod_{i \in I} \cC_i) \to \prod_{i \in I} \pi_n\bfK(\cC_i)$ as a retract of the middle horizontal homomorphism.
 
 Since $F^q_7(\prod_{i \in I} \cC_i) \simeq \prod_{i \in I} F^q_7\cC_i$ and $\bigoplus_{i = 0}^7 (\prod_{i \in I} \cC_i) \simeq \prod_{i \in I} \bigoplus_{i = 0}^7 \cC_i$,
 we see that $B^q_7(\prod_{i \in I} \cC_i) \simeq \prod_{i \in I} B^q_7\cC_i$ because limits commute with each other.
 We conclude that
 \[ \grayson_7(\prod_{i \in I} \cC_i) \simeq \prod_{i \in I} \grayson_7\cC_i. \]
 Now it is immediate from \cref{lem:k_0-products} that $\phi$ is an isomorphism, which implies the claim for $n \geq 1$.
 
 We are left with the case $n < 0$.
 Recall from \cref{sec:lower-k-theory} that $\pi_{-n}\bfK(\cC)$ is naturally isomorphic to $K_0(\Idem(\Sigma^n\cC))$ for $n \geq 1$.
 
 Note that $\prod_{i\in I}\Ind_{\aleph_0}(\cC_i)^\kappa$ and $\Ind_{\aleph_0}(\prod_{i \in I}\cC_i)^\kappa$ have trivial $K$--theory since they admit infinite coproducts. Hence using \cref{lem:products-exact-sequences}, it is easy to see that
 \[ \bfK(\prod_{i \in I} \Sigma^n\cC_i) \simeq \bfK(\Sigma^n \prod_{i \in I} \cC_i), \]
 and the claim follows by another application of \cref{lem:k_0-products}.
\end{proof}

\begin{proof}[Proof of \cref{thm:commute}]
 The theorem is an immediate consequence of \cref{prop:commute-pi_n} and \cite[Remark~1.4.3.8]{LurieHA}.
 The claim about connective $K$--theory follows by applying $\tau_{\geq 0}$; since $\tau_{\geq 0}$ is a right adjoint, it preserves products.
\end{proof}

\begin{proof}[Proof of \cref{thm:uadd-products}]
 By \cite[Theorem~7.13]{BGT13}, we have for every stable $\infty$--category $\cA$ and compact idempotent complete stable $\infty$--category $\cB$ a natural equivalence of spectra
 \[ \Map(\uadd(\cB), \uadd(\cA)) \simeq \Kcon(\Fun^{\mathrm{ex}}(\cB, \Idem(\cA))). \]
 Let now $\{\cC_i\}_{i \in I}$ be any family of stable $\infty$--categories, and let $\cB$ be a compact, idempotent complete stable $\infty$--category.
 Then we have
 \begin{eqnarray*}
  \Map(\uadd(\cB), \uadd(\prod_{i \in I} \cC_i))
  &\simeq& \Kcon(\Fun^{\mathrm{ex}}(\cB, \Idem(\prod_{i \in I} \cC_i))) \\
  &\stackrel{(*)}{\simeq}& \Kcon(\Fun^{\mathrm{ex}}(\cB, \prod_{i \in I} \Idem(\cC_i))) \\
  &\simeq& \Kcon(\prod_{i \in I} \Fun^{\mathrm{ex}}(\cB, \Idem(\cC_i))) \\
  &\stackrel{(**)}{\simeq}& \prod_{i \in I} \Kcon(\Fun^{\mathrm{ex}}(\cB, \Idem(\cC_i))) \\
  &\simeq& \prod_{i \in I} \Map(\uadd(\cB), \uadd(\cC_i))
 \end{eqnarray*}
 where $(*)$ follows from \cref{lem:idem} and $(**)$ follows from \cref{thm:commute}. Since $\madd$ is a localization of $\operatorname{Pre}_{\Sp}((\catperf)^\omega)$ by \cite[Remark~6.8]{BGT13}, it is generated by the images of compact idempotent complete stable $\infty$--categories under $\uadd$. 
 This verifies the universal property of the product. Hence
 \[ \uadd(\prod_{i \in I} \cC_i) \simeq \prod_{i \in I} \uadd(\cC_i).\qedhere \]
\end{proof}

\begin{rem}
	The proof of \cref{thm:uadd-products} breaks down for $\uloc$ since the identification of mapping spectra in $\mloc$ holds only under stricter assumptions on $\cB$, cf.~\cite[Theorem~9.36]{BGT13}.
\end{rem}

\bibliographystyle{alpha}
\bibliography{grayson}

\end{document}